\documentclass[12pt]{amsart}

\usepackage{amsmath,amsthm,amsfonts,amssymb}

\catcode`\@=11

\def\theenumi{\textup{(\@roman\c@enumi)}}

\theoremstyle{plain}                       
\newtheorem{theorem}{Theorem}          
\newtheorem{lemma}[theorem]{Lemma}         
\newtheorem{corollary}[theorem]{Corollary} 

\theoremstyle{remark}                      
\newtheorem*{acknow}{Acknowledgements}          
\newtheorem*{remark}{Remark}          
\newtheorem*{example}{Example}          

\newcommand\sm{\setminus}
\newcommand\col{\colon}
\newcommand\sub{\subseteq}
\newcommand\bus{\supseteq}

\newcommand{\set}[2]{\{\,{\textstyle#1};\,{\textstyle #2}\,\}}

\newcommand{\inv}{^{-1}}                   

\newcommand\Z{\mathbb{Z}}
\newcommand\R{\mathbb{R}}
\renewcommand\H{\mathbb{H}}
\newcommand\conj{\overline}

\newcommand\latetilde{\widetilde{\phantom{x}}}

\DeclareMathOperator\re{Re}
\DeclareMathOperator\im{Im}

\begin{document}
\title[Quasi-actions and rough Cayley graphs]%
{Quasi-actions and rough Cayley graphs of locally compact groups}

\author{Pekka Salmi}

\address{\small Department of Mathematical Sciences, University of Oulu,
         PL~3000, FI-90014 Oulun yliopisto, Finland}

\email{pekka.salmi@iki.fi}

\subjclass[2010]{Primary 20F65; Secondary 05C25, 22D05, 43A07}

\keywords{Cayley graph, locally compact group, quasi-action, quasi-lattice}

\begin{abstract}
We define the notion of rough Cayley graph for 
compactly generated locally compact groups in terms 
of quasi-actions. We construct such a graph for 
any compactly generated locally compact group 
using quasi-lattices 
and show uniqueness up to quasi-isometry. 
A class of examples is given by the Cayley
graphs of cocompact lattices in compactly generated
groups. As an application, we show 
that a compactly generated group
has polynomial growth if and only if its rough Cayley graph
has polynomial growth (same for intermediate and exponential 
growth). Moreover, a unimodular compactly generated group 
is amenable if and only if its rough Cayley graph 
is amenable as a metric space.
\end{abstract}

\maketitle

\section{Introduction}

The purpose of this paper is to introduce the notion 
of rough Cayley graph for compactly generated locally compact 
groups and indicate its usefulness in abstract harmonic analysis.
The definition is based on 
the notion of quasi-action of a locally compact 
group on a metric space. Similar concepts have been 
fruitfully employed in the study of geometric group theory
in the case of discrete groups (see for example 
\cite{de-la-harpe:ggt, msw:quasi-actions-on-trees}),
and one would expect that the corresponding
concepts will be useful in the study of locally compact 
groups. One can study coarse geometry of compactly generated locally compact
groups using the word-length metric with respect to a compact generating set,
as done for example in 
\cite{cornulier-tessera:contracting-automorphisms, 
baumgartner-et-al:hyperbolic}. 
The rough Cayley graph ignores the local information but retains the
large-scale information thereby giving an alternative, but equivalent,
approach to coarse geometry of locally compact groups.

In \cite{kron-moller:rough-cayley}
Kr\"on and M\"oller defined 
the rough Cayley graph of a topological group $G$
to be a connected  
graph $X$ such that $G$ acts transitively on the set of 
vertices of $X$ and 
the stabilisers of the vertices are compact open subgroups of $G$.
We shall follow the terminology of \cite{baumgartner-et-al:hyperbolic}
where these graphs were renamed as relative Cayley graphs. 
Every totally disconnected, compactly generated, locally compact group
$G$ has a locally finite relative Cayley graph:
the vertex set of such a graph can be realised as the homogeneous 
space $G/U$ with respect to a compact open subgroup $U$.
In \cite{baumgartner-et-al:hyperbolic}
relative Cayley graphs were used to show that hyperbolic 
totally disconnected groups have flat rank at most $1$.
In the case of totally disconnected, compactly generated, locally
compact groups, the relative Cayley graph fits into our 
notion of rough Cayley graph. 
To give further examples of rough Cayley graphs, we shall show that the
Cayley graph of a cocompact lattice in a compactly generated group
is the rough Cayley graph of the ambient group. 
We shall also construct 
the rough Cayley graph of 
the affine group of the real line.
Our constructions of rough Cayley graphs, both in the abstract 
setting and in the example of the affine group, make use of 
quasi-lattices. 

As an application of rough Cayley graphs, we shall show that the
growth of a compactly  generated group can be
described in terms  of its rough Cayley graph. Moreover, we shall 
show that when the compactly generated group is unimodular, it is
amenable  if and only if its rough Cayley graph is amenable in the
sense  of \cite{block-weinberger:amenability}.
The example of the affine group 
shows that unimodularity is a necessary assumption for this result.

\section{Quasi-action}

We start with some terminology from coarse geometry.
When $X$ is a metric space and $A\sub X$, we write 
\[
N_r(A) = \set{x\in X}{d(x,a)\le r\text{ for some }a\in A}
\]
(and $N_r(x):= N_r(\{x\})$ for $x\in X$).
Let $C\ge 1$ and $r\ge 0$ be constants. 
A map $f\col X\to Y$ between metric spaces is 
a $(C,r)$\emph{-quasi-isometric embedding} if
\[ 
C\inv d_X(x,y) - r\le d_Y(f(x),f(y)) \le C d_X(x,y) + r
\]
for every $x,y\in X$. 
In this case, $f$ is a $(C,r)$\emph{-quasi-isometry} 
if $f$ is also $r$-\emph{coarsely onto}: 
$Y = N_r(f(X))$.
A \emph{coarse inverse} of a quasi-isometry $f\col X\to Y$ 
is a quasi-isometry $g\col Y\to X$ such that 
for some $r\ge 0$, we have
$d(g\circ f(x), x)\le r$ and $d(f\circ g(y),y)\le r$ 
for every $x\in X$ and $y\in Y$. 

A $(C,r)$-\emph{quasi-action} of a locally compact 
group $G$ on a metric space  $X$ is a map
$(s,x)\mapsto s\cdot x\col G\times X\to X$  satisfying the following conditions:
\begin{enumerate}
\item $x\mapsto s\cdot x\col X\to X$ is a $(C,r)$-quasi-isometry
      for every $s\in G$ 
\item  $d(e\cdot x,x) \le r$ for every $x\in X$, 
       where $e$ denotes the identity of~$G$ 
\item $d(s\cdot (t\cdot x), (st)\cdot x)\le r$
       for every $s,t\in G$ and $x\in X$
\item $K\cdot x$ is bounded whenever $K\sub G$ is compact and 
       $x\in X$.
\end{enumerate}
In the case of discrete groups, 
the notion of quasi-action has been used for
example in \cite{farb-mosher,msw:quasi-actions-on-trees}.
Note that in the discrete case 
condition (iv) is of course unnecessary.

We say that a quasi-action is 
\emph{cobounded}
if there is $r\ge 0$ 
such that for every $x\in X$ the map $s\mapsto s\cdot x\col G\to X$ 
is $r$-coarsely onto
(it might be more descriptive to call such quasi-actions 
coarsely transitive, but we follow the 
the terminology of \cite{msw:quasi-actions-on-trees}).
To simplify the notation, we shall 
use the same `$r$' to denote 
the constant associated with 
a $(C, r)$-quasi-action 
as well as the constant associated with coboundeness.

We say that a quasi-action is \emph{proper} if for every $R\ge 0$
and for every $x\in X$ the set
\[
\set{s\in G}{d(s\cdot x, x) \le R } 
\]
is relatively compact. 
Note that if $X$ is a proper metric space
(i.e.\ all closed, bounded sets in the metric space $X$ are compact),
a quasi-action is proper if and only if for every compact set $K\sub
X$ the set $\set{s\in G}{s\cdot K \cap K\ne \emptyset}$
is relatively compact (this latter condition agrees with the 
usual definition of proper action 
as in \cite[p. 87]{de-la-harpe:ggt}).

`Compactly generated group' refers to a compactly generated
\emph{locally compact} group. 
We consider compactly generated groups as metric spaces equipped with 
the left-invariant word-length metric with respect to a generating set that
is a compact symmetric neighbourhood of the identity. 
Given two compact generating neighbourhoods, one is contained to 
some power of the other (and vice versa), so 
the choice of the generating set does not affect the 
quasi-isometric class of the metric space,
that is, two different generating sets give quasi-isometric spaces.

Let $G$ be a compactly generated group.
We say that two quasi-actions 
$(s,x)\mapsto s\cdot x\col G\times X\to X$
and  $(s,y)\mapsto s\cdot y\col G\times Y\to Y$
on metric spaces $X$ and $Y$ are \emph{quasi-conjugate} if 
there is a quasi-isometry $\phi\col X\to Y$ and a constant 
$R\ge 0$ such that 
\[
d_Y( \phi(s\cdot x), s\cdot \phi(x) )\le R
\]
for every $s\in G$ and $x\in X$.
In the following lemma, which generalises Proposition~2.1
of \cite{farb-mosher} to our situation, 
we define a quasi-action on any metric space to which 
$G$ is quasi-isometric. The quasi-action is 
defined by conjugating the action of $G$ on itself
with a quasi-isometry, and so the resulting quasi-action is
naturally quasi-conjugate to the action of $G$ on itself.

\begin{lemma} \label{lemma:q-iso->q-action}
Let $G$ be a compactly generated group and 
$X$ a metric space. 
Suppose that $\phi\col G \to X$ is a quasi-isometry 
and $\psi\col X \to G$ is its coarse inverse.
Then $s\cdot x = \phi (s\psi(x))$, for $s\in G$ and $x\in X$, 
defines a proper cobounded quasi-action of $G$ on $X$.
\end{lemma}

\begin{proof}
It is easy to check that we indeed have a quasi-action. 
Coboundedness is also immediate because $\phi$ is coarsely onto
and $G\psi(x) = G$. To check that the quasi-action is proper,
fix $R>0$ and $x\in X$. Then $d(s\cdot x, x)\le R$ implies that 
\begin{align*}
R&\ge d(\phi(s\psi(x)), x) \ge 
d\bigl( \phi(s\psi(x)), \phi(\psi(x)) \bigr) - d( \phi(\psi(x)), x) \\
&\ge C\inv d\bigl( s\psi(x), \psi(x) \bigr) - r
\end{align*}
for some constants $C\ge 1$ and $r\ge 0$. It follows that 
for some integer $m$, we have $s\in \psi(x) K^m \psi(x)\inv$.
Since the same $m$ works for all $s\in G$ such that 
$d(s\cdot x, x)\le R$, the quasi-action is proper. 
\end{proof}

The following result is an analogue of the \v Svarc--Milnor lemma
for quasi-actions of locally compact groups 
(for the classical \v Svarc--Milnor lemma, see for example 
\cite[Theorem IV.B.23]{de-la-harpe:ggt}). 
Following \cite{abels-margulis}, we say that a metric space $X$ 
is $c$-\emph{coarsely geodesic}, with $c\ge 0$, if for every $x,y\in X$ 
there is $f\col [0, a] \to X$ such that 
$f(0) = x$, $f(a) = y$ and 
\[
|s-t| - c \le d(f(s), f(t))\le |s-t| + c
\]
for every $s,t\in [0,a]$. The map $f$ is called a 
$c$-\emph{coarse geodesic} from $x$ to $y$.

\begin{lemma} \label{lemma:svarc-milnor}
Suppose that $G$ is a locally compact group that 
has a proper cobounded quasi-action $(s,x)\mapsto 
s\cdot x\col G\times X\to X$ on a coarsely geodesic metric space $X$. 
Then $G$ is compactly generated and for any fixed $x\in X$ 
the map $s\mapsto s\cdot x\col G\to X$ is a quasi-isometry.
Moreover, the quasi-action of $G$ on $X$ is 
quasi-conjugate to the left action of $G$ on itself. 
\end{lemma}

\begin{proof}
Let $C\ge 1$, $r\ge 0$ and $c\ge 0$ be constants
such that $X$ is $c$-coarsely geodesic and 
$G$ has a $(C,r)$-quasi-action on $X$.
Fix $x\in X$. 
Choose $R = C(2r+ c + 1) + 4r$ and let $K$ 
be a compact symmetric neighbourhood of the identity containing
\[
\set{s\in G}{ d(s\cdot x, x) \le R}
\]
($K$ exists because the quasi-action is proper).
We shall first show that $K$ generates $G$. 

Let $u\in G\sm K$ and let $n\ge 2$ be the unique integer such that 
\[
R - r -c + n-2 \le d(u\cdot x, x) < R - r - c + n-1.
\]
Let $f\col [0,a]\to X$ be a $c$-coarse geodesic from $x$ to 
$u\cdot x$, and put $y_0 = x$, $y_n = u\cdot x$ and 
$y_i = f(R-r-c+i-1)$ for $i = 1, 2, \ldots, n-1$.
Then put $s_n = u$ and for each $i=1$, \ldots,
$n-1$ choose $s_i\in G$ such that $y_i \in N_{r}(s_i\cdot x)$,
using coboundedness. Then $s_1\in K$ and for every $i=1$, \ldots, $n-1$
\begin{align*}
d(s_i\inv s_{i+1}\cdot x, x) \le C  d(s_{i+1}\cdot x, s_i\cdot x) + 4r
\le C (2r+ c +1) + 4r = R
\end{align*}
and so $s_i\inv s_{i+1}\in K$. Therefore 
\[
u = s_n = s_1 (s_1\inv s_2)(s_2\inv s_3)\ldots (s_{n-1}\inv s_n)
\in K^n.
\]
It follows that $K$ generates $G$, so we may use the word-length metric 
on $G$ with respect to $K$. Another consequence of the preceding
calculation is that for every $s,t\in G$ with $s\inv t\notin K$  
there is $n\ge 2$ such that $s\inv t\in K^n$ and 
\[
n \le d(s\inv t\cdot x, x) + 2 + r +c - R.
\]
Hence
\[
d_G(s, t) \le C  d(s\cdot x, t\cdot x) + 3r + 2 + r +c - R.
\]
Including the case when $s\inv t\in K$, we have 
\[
d_G(s,t) \le C  d(s\cdot x, t\cdot x) + 1
\]
for every $s,t\in G$.

Let 
\[
M = \sup\set{d(s\cdot x, x)}{s\in K},
\]
which exists because $K$ is compact. For every $s_1$, $s_2$, \ldots,
$s_n$ in $K$, we have
\begin{align*}
d(s_1\ldots s_n\cdot x, x)
&\le d(s_1 \cdot x, x) + d(s_1s_2\cdot x, s_1\cdot x) + \cdots \\
 &\quad   +d(s_1\ldots s_n\cdot x, s_1\ldots s_{n-1} \cdot x) \\
&\le M + (n-1)(C M + 2r) 
\end{align*}
Suppose that $d_G(s,t) = n$ so that $s\inv t\in K^n$.
Then 
\[
d(s\cdot x, t\cdot x) \le C  d(s\inv t\cdot x, x) + 4C r 
\le C (C M + 2r) d_G(s,t) + 4 C r
\]
by the preceding calculation.
As the quasi-action is cobounded, the map $\phi\col s\mapsto s\cdot x$ 
is coarsely onto, and so it is a quasi-isometry.

Finally, for every $s,t\in G$
\[
d(s\cdot \phi(t), \phi(st)) = d(s\cdot(t\cdot x), (st)\cdot x),
\]
so the action of $G$ on itself is quasi-conjugate 
to the quasi-action of $G$ on $X$, as the latter is coarsely
associative.  
\end{proof}

\begin{remark} 
Suppose that the topology of a locally compact group $G$ is 
induced by  a left-invariant metric $d$. 
If the metric space $(G,d)$ is coarsely geodesic and proper,
we may apply Lemma~\ref{lemma:svarc-milnor} 
to the left action of $G$ on itself 
and see that $(G,d)$ is quasi-isometric to 
$G$ equipped with the word-length metric.
\end{remark}

\section{Rough Cayley graph} \label{sec:rough}

In this section we shall give the
formal definition of rough Cayley graph
and consider two classes of examples 
of rough Cayley graphs, namely, the so-called relative Cayley 
graphs of totally disconnected compactly generated groups 
and the usual Cayley graphs of cocompact lattices in compactly 
generated groups.
Both classes include the Cayley graphs of finitely generated discrete 
groups. As a further example not belonging to either of these classes,
we shall construct the rough Cayley graph of the affine group of the
real line. 

We consider graphs as discrete metric spaces consisting 
of vertices and equipped with the graph metric.
A graph is said to be \emph{uniformly locally finite}
if there is a uniform bound on the degree of vertices.

Let $G$ be a compactly generated locally compact group.
A \emph{rough Cayley graph} of $G$ is a uniformly locally finite, 
connected graph $X$ such that $G$ has a proper cobounded
quasi-action on $X$.

The notion of quasi-lattice is important for 
our constructions of rough Cayley graphs. 
Let $X$ be a metric space. 
A  subset $A\sub X$ is $r$-\emph{coarsely dense} in $X$ 
if $N_r(A) = X$. A \emph{quasi-lattice} in $X$ 
is a coarsely dense subset $\Gamma\sub X$ such that 
for every $R>0$ there is $M>0$ such that
\[
|\Gamma\cap N_R(x)|\le M
\]
for every $x\in X$.

\begin{lemma} \label{lemma:q-lat}
Let $X$ be a coarsely geodesic metric space and 
let $\Gamma$ be a quasi-lattice in $X$. 
Then there is a graph structure on $\Gamma$ that 
makes $\Gamma$ a  uniformly locally finite, connected
graph that is quasi-isometric to $X$.
\end{lemma}

\begin{proof}
Fix $c\ge 0$ such that $X$ is $c$-coarsely geodesic
and $r\ge 0$  such that $X = N_r(\Gamma)$. 
We declare that vertices $x, y\in \Gamma$,  $x\ne y$, are adjacent if 
$d(x,y) \le 2r + c + 1$. 
We denote the resulting graph metric on $\Gamma$ 
by $d_\Gamma$, but of course we still need to check that 
$\Gamma$ is in fact a connected graph so that 
$d_\Gamma$ is well defined. 

Given $x,y\in \Gamma$, let $f\col [0,a]\to X$ 
be a $c$-coarse geodesic from $x$ to $y$. 
Similarly to the proof of Lemma~\ref{lemma:svarc-milnor},
put $y_k = f(k)$ for integers $0<k<a$ and
pick $x_k\in \Gamma\cap N_r(y_k)$. Put 
$x_0 = x$ and $x_n = y$ where $n = \lceil a \rceil$.
Then for every $k = 0$, $1$, \ldots, $n-1$
\[
d(x_k, x_{k+1}) \le 2r + c + 1,
\]
and so $x_k$ and $x_{k+1}$ are adjacent. 
This shows that $\Gamma$ is connected as a graph, 
and morever, that
\[
d_\Gamma(x,y)\le n < a+1\le d(x,y) + c + 1.
\]
On the other hand, the definition of the graph metric $d_\Gamma$ 
immediately yields
\[
d(x,y) \le (2r + c +1) d_\Gamma(x,y).
\]
It follows that $\Gamma$ equipped 
with the graph metric is quasi-isometric 
to $X$.

As $\Gamma$ is a quasi-lattice in $X$, 
there is $M>0$ such that 
\[
|\Gamma\cap N_{2r+c+1}(x)|\le M
\]
for every $x\in X$. In particular, $M$ is a uniform bound 
on the degree of vertices, and so $\Gamma$ is uniformly 
locally finite.
\end{proof}

A set $T\sub G$ is \emph{right uniformly discrete} 
with respect to a relatively compact 
neighbourhood $V$ of the identity 
if $t V \cap t' V = \emptyset$ whenever $t\ne t'$ are in $T$.
A right uniformly discrete subset is maximal if it is not properly
contained to another subset that is right uniformly discrete 
with respect to $V$ (this is equivalent to $TVV\inv = G$).

The following result gives the existence and the uniqueness of rough 
Cayley graph. It should be compared with 
Theorems~2.2 and 2.7+ of \cite{kron-moller:rough-cayley} 
that give the existence and the uniqueness of relative Cayley graphs of 
totally disconnected compactly generated groups, but note that 
our argument is completely different from those in 
\cite{kron-moller:rough-cayley}.

\begin{theorem} \label{thm:exists}
There exists a rough Cayley graph for any 
compactly generated locally compact group $G$.
The rough Cayley graph of $G$ is unique in the sense that
the quasi-actions of $G$ on two of its rough Cayley graphs are
quasi-conjugate.  
\end{theorem}

\begin{proof}
Let $V$ be a  relatively compact neighbourhood of 
the identity. 
As $G$ is $\sigma$-compact we may construct 
a  maximal right uniformly discrete set $X\sub G$ 
with respect to $V$ using induction
(in general, maximal right uniformly discrete sets exist by Zorn's lemma).
We shall first show that $X$ is a quasi-lattice in $G$.

Fix a compact symmetric neighbourhood $K$ of the identity 
such that $K$ generates $G$, and consider $G$ as a metric space 
equipped with the word-length metric with respect to $K$. 
There is an integer $n$ such that $VV^{-1}\sub K^n$, and so
$X$ is $n$-dense in $G$.
For any given $m>0$, the set $K^m$ can be covered by finitely 
many, say $M$, right translates of $V$.
Fix $y\in G$. Since $X$ is right uniformly discrete with respect 
to $V$, no right translate of $V$ may contain more than one 
point of the form $x\inv y$ with $x\in X$. Therefore
$|X\cap yK^m|\le M$, and so $X$ is a quasi-lattice.

By Lemma~\ref{lemma:q-lat} we can give $X$ a graph structure
such that $X$ becomes a uniformly locally finite, connected 
graph. Moreover, $X$ equipped with the graph metric 
is quasi-isometric with $G$. By Lemma~\ref{lemma:q-iso->q-action}
there is a cobounded proper quasi-action of $G$ on $X$. 
Consequently, $X$ is a rough Cayley graph of $G$.

If $X$ and $Y$ are two rough Cayley graphs of $G$,
then the quasi-actions of $G$ on $X$ and $Y$ are 
both quasi-conjugate to the left action of $G$ on itself
by Lemma~\ref{lemma:svarc-milnor}. 
Therefore the quasi-actions of $G$ on $X$ and $Y$ are 
quasi-conjugate. 
\end{proof}

Recall from the introduction the 
notion of relative Cayley graph of a topological group $G$ 
due to Kr\"on and M\"oller 
\cite{kron-moller:rough-cayley,baumgartner-et-al:hyperbolic}:
a connected graph $X$ such that $G$ acts transitively 
on the set of vertices of $X$ and 
the stabilisers of the vertices are compact open subgroups of $G$.
When $G$ is a totally disconnected compactly generated group,
its relative Cayley graph, which in this case is locally finite,
may be realised as the homogeneous space $G/U$ 
where $U$ is  a compact open subgroup
(see \cite[p.~643]{kron-moller:rough-cayley}).
It is not very difficult to show that the graph
is a rough Cayley graph in our sense, that is, the action 
is cobounded and  proper.
Then by Theorem~\ref{thm:exists} any rough Cayley graph 
of a totally disconnected compactly generated 
group is quasi-conjugate to a relative Cayley graph,  
and so the quasi-action is in this case quasi-conjugate to 
an isometric action.

The following result generalises Corollary~2.11 of 
\cite{kron-moller:rough-cayley} concerning relative Cayley 
graphs. Here we say that a subgroup $H$ of $G$ is \emph{cocompact}
if the quotient space $G/H$ is compact. 
The first statement of the result is known. 

\begin{theorem}
Suppose that  $G$ is a compactly generated group and 
$H$ is a closed cocompact subgroup of $G$. 
Then $H$ is compactly generated and the rough Cayley 
graph of $H$ may be viewed as the rough Cayley graph of $G$. 
\end{theorem}

\begin{proof}
Consider the left action of $H$ on $G$ equipped with 
the word-length metric. As $H$ is cocompact, the action is cobounded. 
Given $x\in G$, the set
\[
\set{h\in H}{d(hx,x)\le R} = N_R(x)x\inv \cap H
\]
is compact, so the action is also proper. 
Lemma~\ref{lemma:svarc-milnor} implies that $H$ is 
compactly generated and quasi-isometric to $G$.
This in turn implies that 
if $X$ is any rough Cayley graph of $H$,
then $G$ has the necessary quasi-action on $X$ by 
Lemma~\ref{lemma:q-iso->q-action}.
\end{proof}

As a special case of the preceding result, 
we see that if a compactly generated group $G$ has a 
cocompact lattice $\Gamma$ 
(i.e.\ a cocompact discrete subgroup),
then the Cayley graph of $\Gamma$ is 
the rough Cayley graph of $G$.

\begin{example}
Consider the affine group $G = \R\rtimes \R_+$ of the real line. 
The multiplication of $G$ is defined by
\[
(u,a)(v,b) = (av + u, ab) \qquad u,v\in \R, a,b\in \R_+.
\]
Let $d$ be the metric on $G$ induced by the 
Riemannian metric $ds^2 = a^{-2}(du^2 + da^2)$;
that is, $d$ is the metric of the hyperbolic 
plane $\H^2 = \set{(u,a)\in\R^2}{a>0}$. 
Note that $d$ left-invariant and that 
$G$ with the word-length metric is quasi-isometric
to $(G,d)$ by the remark following Lemma~\ref{lemma:svarc-milnor}. 
Interpreting the elements of $G$ as complex numbers, we have the formulas 
\begin{equation} \label{eq:hyp-d}
d(x, y) = 
2\tanh\inv\biggl|\frac{x - y}{x - \conj{y}}\biggr|
= 
\log \biggl(\frac{|x - \conj{y}| + |x - y|}%
{|x - \conj{y}| - |x - y|}\biggr) 
\end{equation}
for $x,y\in G$.

Put
\[
X = \set{(e^n m,e^n)}{n,m\in \Z}\sub G.
\]
To show that $X$ determines the rough Cayley graph of 
$G$, it suffices to show by 
Lemmas~\ref{lemma:q-iso->q-action} and~\ref{lemma:q-lat} 
that $X$ is a quasi-lattice in $(G,d)$.

Let $y = (u,a)\in G$ be arbitrary.
Choose $n\in\Z$ such that $|n-\log a|\le 1/2$ 
and then $m\in\Z$ such that $|e^{-n}u - m|\le 1/2$. 
Put $x = (e^n m, e^n)$. Then $e^{-1/2} a  \le e^n \le e^{1/2} a$ and we have
\[
|x - y|\le e^n\sqrt{\frac14 + (e^{1/2} - 1)^2}.
\]
On the other hand
\[
|x - \conj{y}|\ge e^n\sqrt{\frac14 + (e^{-1/2} + 1)^2}.
\]
It follows that 
\[
d(x,y) \le 2 \tanh\inv \sqrt{\frac{\frac14 + (e^{1/2} - 1)^2}%
{\frac14 + (e^{-1/2} + 1)^2}} \approx 1.06.
\]
This shows that $X$ is coarsely dense in $(G,d)$.

Now suppose that $x = (e^{n}m, e^{n})\in X$ and 
$y = (u, a)\in G$ satisfy $d(x,y)\le R$. 
Write 
\begin{align*}
r &= \re(x - \conj{y}) = \re(x - y)
= e^{n}m - u \\
I_1 &= \im(x - \conj{y}) = e^{n} + a \\
I_2 &= \im(x - y) = e^{n} - a.
\end{align*}
As $d(x,y)\le R$, 
\[
(|x - \conj{y}| + |x - y|)^2 \le e^R
(|x - \conj{y}|^2 - |x - y|^2)
\]
and so
\[
I_1^2 + I_2^2 + 2r^2 + 
2\sqrt{(I_1^2 + r^2)(I_2^2 + r^2)}\le e^R(I_1^2 - I_2^2).
\]
Hence
\[
(I_1 + I_2)^2 + 2r^2 
\le e^R(I_1^2 - I_2^2)
\]
and so
\[
I_1 + I_2 + \frac{2r^2}{I_1+I_2} 
\le e^R(I_1 - I_2).
\]
That is
\begin{equation} \label{eq:n_1}
2e^{n}  + r^2 e^{-n}  
\le 2e^{R}a.
\end{equation}
Therefore $n \le \log a + R$, 
and by symmetry $|n - \log a|\le R$.

Assuming that $y = (u,a)$ is fixed, 
let $n = \log a + k$ be an integer such that $|k|\le R$.
Then \eqref{eq:n_1} implies that 
\[
(e^{k}a m - u)^2 \le 2 a^2 e^{k}(e^R - e^{k}),
\]
and it follows that 
\[
|m - e^{-k} a\inv u| \le \sqrt{2(e^{2R}- 1)}.
\]
So we obtain a uniform bound on the number of 
possible choices of $x = (e^nm,e^n)$ such that $d(x, y)\le R$
once $y$ is fixed. Therefore $X$ is a quasi-lattice in $(G,d)$ as required.
\end{example}

\section{Growth} \label{sec:growth}

In this section, we shall show that compactly generated groups 
have the same growth as the associated rough Cayley graphs.
In the setting of relative Cayley graphs,
Kr\"on and M\"oller obtained this result for 
totally disconnected compactly generated groups in Theorem~4.4 of
\cite{kron-moller:rough-cayley}.

Let $G$ be a locally compact group generated  
by a compact symmetric neighbourhood $K$, and denote the 
left Haar measure of $G$ by~$\lambda$. 
Recall that $G$ has \emph{polynomial growth} if $\lambda(K^m)$ is 
bounded by a polynomial in $m$,
\emph{intermediate growth} if it does not have polynomial growth 
but $\limsup \lambda(K^m)^{1/m} = 1$,
and \emph{exponential growth} otherwise. 
The definitions are independent on the choice of $K$.
The growth of a connected locally finite graph is defined similarly, 
replacing $\lambda(K^m)$ by the cardinality of vertices in $N_m(x_0)$,  
where $x_0$ is a fixed base point.
Note that the growth does not depend on the chosen base point,
although the actual values of $|N_m(x_0)|$ may.

We now fix notation for the rest of the section.
Let $X$ be a rough Cayley graph of a compactly generated group $G$. 
Let $C_1\ge 1$ and $r_1\ge 0$ be the constants 
associated with the quasi-action of $G$ on $X$. 
Let $K\sub G$ be a compact symmetric neighbourhood of the 
identity such that $K$ generates $G$, and
consider $G$ a metric space with the word-length metric 
defined with respect to $K$. Fix a base point $x_0\in X$ and  
let $C_2\ge 1$ and $r_2\ge 0$ be constants such that $s\mapsto s\cdot x_0$ 
is a $(C_2,r_2)$-quasi-isometry.
Finally, let $T$ be a maximal right uniformly discrete subset of $G$ 
with respect to $K$, and recall that $G = TK^2$. 

\begin{lemma} \label{lemma:T-dense}
For $r := 2C_2 + r_1 + r_2$, we have $X = N_{r}(T\cdot x_0)$.
\end{lemma}

\begin{proof}
Let $y\in X$ be arbitrary. Then there is $s\in G$ such that
$d(s\cdot x_0, y)\le r_1$. But $G = TK^2$ implies that there 
exists $t\in T$ such that $d_G(s,t) \le 2$. 
Hence $d(t\cdot x_0,y)\le r_1 + d(s\cdot x_0, t\cdot x_0) \le r$.
\end{proof}

For every $A\sub X$, define 
\[
\widetilde{A} = \set{t\in T}{N_{r}(t\cdot x_0)\cap A\ne \emptyset}.
\]

To simplify notation, write $C := C_2$; 
moreover, we assume that $C$ and $r$ are integers.

\begin{lemma}  \label{lemma:nuggets}
Let $A\sub X$.
\begin{enumerate}
\item $\widetilde{A}$ is right uniformly discrete with respect to $V$.
\item $M^{-r}|A| \le |\widetilde{A}| \le M^{r}|A|$ where 
      $M$ is a uniform bound on the degree of vertices in $X$. 
\item If $s\cdot x_0 \in A$, then $s\in \widetilde{A}K^{2Cr}$.
\item If $s\in \widetilde{A}K^m$, then $d(s\cdot x_0, A)\le C m+2r$.
\end{enumerate}
\end{lemma}

\begin{proof}
The first statement is immediate from the definition of $\widetilde{A}$.

The second statement follows from the inclusions 
$\widetilde{A}\cdot x_0\sub N_{r}(A)$ and 
$A\sub N_{r}(\widetilde{A}\cdot x_0)$ 
(the latter inclusion by Lemma~\ref{lemma:T-dense}).

The third statement holds because 
$d(s\cdot x_0,t\cdot x_0)\le r$ for some $t\in \widetilde{A}$ 
(by Lemma~\ref{lemma:T-dense}), 
and hence $d_G(s,t) \le C d(s\cdot x_0, t\cdot x_0) + C r_2 \le 2Cr$.

As for the fourth statement, 
now $d_G(s,t)\le m$ for some $t\in\widetilde{A}$, so
$d(s\cdot x_0, A) \le C m + r_2 + r\le Cm + 2r$. 
\end{proof}

\begin{theorem}
A compactly generated group and its rough Cayley graph have the same growth. 
\end{theorem}

\begin{proof}
Let $m$ be a positive integer and suppose that 
$s\in K^{m}$. Then $d(s\cdot x_0, x_0) 
\le d(s\cdot x_0, e\cdot x_0) + r_1 \le C m + r$
and so $s\in N_{Cm+r}(x_0)\latetilde K^{2Cr}$ 
by Lemma~\ref{lemma:nuggets}.
Therefore $K^{m}\sub N_{C m+r}(x_0)\latetilde K^{2Cr}$ and so
\begin{equation} \label{eq:1st-approx-A}
\lambda(K^m) \le |N_{C m+r}(x_0)\latetilde|\,
   \lambda(K^{2Cr}) \le M^{r}\lambda(K^{2Cr})|N_{C m+r}(x_0)|.
\end{equation}
Conversely, if $s\in N_m(x_0)\latetilde$, then 
$d(s\cdot x_0, x_0)\le m + r$ and so
$d_G(s, e)\le C (m+2r)$. It follows that 
\[
N_m(x_0)\latetilde K \sub K^{C (m+2r) +  1},
\]
and hence
\begin{equation} \label{eq:2nd-approx-A}
|N_m(x_0)| \le M^r |N_m(x_0)\latetilde| \le  \frac{M^{r}}{\lambda(K)} 
\lambda(K^{C (m + 2r)+ 1}),
\end{equation}
as $N_m(x_0)\latetilde$ is right uniformly discrete with respect to $K$.
Combining \eqref{eq:1st-approx-A} and \eqref{eq:2nd-approx-A} 
we see that $X$ and $G$ have the same growth.
\end{proof}

\section{Amenability} \label{sec:amen}

In this section, we shall relate the amenability of a compactly 
generated group to the amenability of its rough Cayley graph,
where the latter means amenability as a metric space, 
as defined in \cite[section~3]{block-weinberger:amenability}. 
A metric space is of  \emph{coarse bounded 
geometry} if there exists a quasi-lattice in $X$. 
In particular, a uniformly locally finite, connected graph is 
of coarse bounded geometry. 
In a metric space $X$, the $c$-\emph{boundary} of a set $A\sub X$ is 
\[
\partial_c A = \set{x\in X}{d(x,A)\le c\text{ and }d(x,X\sm A)\le c}.
\]
Consider a metric space $X$ of coarse bounded geometry, and let $\Gamma$ 
be a quasi-lattice in $X$. Then $X$ is \emph{amenable}
if for every $c>0$ and $\epsilon>0$ there is a finite set 
$A\sub \Gamma$ such that 
\[
\frac{|(\partial_c A)\cap \Gamma|}{|A|} < \epsilon.
\]
Amenability of metric spaces of coarse bounded
geometry is invariant under quasi-isometries
(and so does not depend on the choice of the quasi-lattice).
Therefore it follows from Lemma~\ref{lemma:svarc-milnor}
that the rough Cayley graph of a compactly generated group $G$ 
is amenable if and only if $G$ is amenable \emph{as a metric space} 
with respect to the word-length metric.
Hence the main result of this section implies that 
a unimodular $G$ is amenable as a metric space if and only 
if it is amenable as a locally compact group.

Recall that a locally compact group $G$ is amenable if and only if 
for every $\epsilon > 0$ and for every compact $F\sub G$ 
containing the identity there is a non-null compact set 
$L\sub G$ such that 
\[
\frac{\lambda(FL\sm L)}{\lambda(L)}<\epsilon.
\]
(This slightly unusual formulation of amenability 
is due to Emerson and Greenleaf \cite{emerson-greenleaf:folner}; 
see also \cite[Theorem~4.13]{paterson:amenability}.) 
It turns out that with our choice
of notation, the right-handed version of the above definition is actually 
the right one: replace $FL$ with $LF$ and the left Haar measure 
$\lambda$ with the right Haar measure. However, we shall only
consider the unimodular  case, so the latter part of the remark may be
ignored: $\lambda$ is both left and right invariant.

\begin{theorem} \label{thm:amen}
Suppose that $G$ is a unimodular compactly generated group. 
The rough Cayley graph of $G$ is amenable if and only 
if $G$ is amenable. 
\end{theorem}

\begin{proof}
We continue with the notation set up in section~\ref{sec:growth}.
Suppose first that the rough Cayley graph $X$ is amenable. 
Since every compact set is contained in some $K^m$,
it is enough to deal with these sets.
Given an integer $m$ and $\epsilon>0$, choose $A\sub X$
such that 
\[
\frac{|\partial_{Cm+2(C ^2+1)r} A|}{|A|} < \epsilon.
\]
Define $L = \widetilde{A}K^{2C r}$.

Claim: 
\[
LK^m\sm L\sub (\partial_{Cm+2(C ^2+1)r} A)\latetilde K^{2C r}.
\]
Let $s \in LK^m\sm L$.
Then  $s\in \widetilde{A}K^{m+2C r}$, and 
so 
\[
d(s\cdot x_0, A) \le Cm + 2(C ^2+1)r
\]
by Lemma~\ref{lemma:nuggets}.
On the other hand, if $s\cdot x_0\in A$, then $s\in L$
by Lemma~\ref{lemma:nuggets}. Therefore $s\cdot x_0\in \partial_{Cm+2(C ^2+1)r} A$ 
and the claim follows again by Lemma~\ref{lemma:nuggets}.

It follows from the claim that
\[
\lambda(LK^m\sm L) \le 
M^{r}\lambda(K^{2C r})|\partial_{Cm+2(C ^2+1)r} A|.
\]
On the other hand, $L \bus \widetilde{A}K$ and 
$\widetilde{A}$ is right uniformly discrete with respect to $K$,
so $\lambda(L)\ge |\widetilde{A}|\lambda(K)$.
Therefore
\[
\frac{\lambda(LK^m\sm L)}{\lambda(L)} \le
\frac{M^{2r}\lambda(K^{2C r})}{\lambda(K)}\, 
 \frac{|\partial_{Cm+2(C ^2+1)r} A|}{|A|}\le
\frac{M^{2r}\lambda(K^{2C r})}{\lambda(K)}\epsilon.
\]
Consequently, $G$ is amenable.

Conversely, suppose that $G$ is amenable. 
Let $\epsilon>0$ and let $m$ be a positive integer.
Write $F = K^{C (2m+3r)+r + 1}$.
Then there is a compact set $L\sub G$ such that 
\[
\frac{\lambda(LF\sm L)}{\lambda(L)} < \epsilon.
\]
Define
\[
A = N_{m + 2r }(L\cdot x_0).
\]
Claim 2: 
\[
(\partial_m  A)\latetilde K \sub LF\sm L.
\]
Let $s\in (\partial_m  A)\latetilde$. Then there is $x\in \partial_m  A$
such that $d(s\cdot x_0, x)\le r$ and hence $t\in L$ 
such that $d(s\cdot x_0, t\cdot x_0) \le 2m + 3r$. 
Therefore $s\in LK^{C (2m+3r)+r}$.  
If $sv\in L$ for some $v\in K$, then 
\[
d(sv\cdot x_0, x) \le  d(sv\cdot x_0, s\cdot x_0) + r \le C_2d(sv,s)+r_2 + r
\le 2r - 1 
\]
and so $d(L\cdot x_0, X\sm A)\le m + 2r - 1$.
This is a contradiction so $sv\notin L$ and claim~2 is proved.

It follows from claim~2 that 
\begin{equation} \label{eq:dA-card}
|\partial_m A| \le  \frac{M^{r}}{\lambda(K)} \lambda(LF \sm L). 
\end{equation}
On the other hand, if $s\in L$, then $s\cdot x_0\in A$
and hence $s\in \widetilde{A}K^{2C r}$.
Therefore
\[
\lambda(L) \le |\widetilde{A}|\lambda(K^{2C r})
\]
and so
\begin{equation} \label{eq:A-card}
|A| \ge \frac{\lambda(L)}{M^{r}\lambda(K^{2C r})}. 
\end{equation}
Combining the approximations \eqref{eq:dA-card} and \eqref{eq:A-card} 
we have
\[
\frac{|\partial_m A|}{|A|} \le  
  \frac{M^{2r}\lambda(K^{2C r})}{\lambda(K)}\,
  \frac{\lambda(LF \sm L)}{\lambda(L)}
\le  \frac{M^{2r}\lambda(K^{2Cr})}{\lambda(K)}\epsilon.
\]
Consequently, $X$ is amenable.
\end{proof}

The unimodularity assumption of the preceding theorem is necessary:
the affine group of the real line is amenable as a locally compact group, 
but as shown in section~\ref{sec:rough}, its rough Cayley 
graph is quasi-isometric to the hyperbolic plane, 
which is not amenable as a metric space
\cite[Example~3.47]{roe:coarse-geometry}.

\begin{corollary}
Suppose that $G_1$ and $G_2$ are unimodular compactly generated 
groups that are quasi-isometric. 
Then $G_1$ is amenable if and only if $G_2$ is amenable.
\end{corollary}

\begin{corollary}
Suppose that $G$ is a compactly generated group
and that $\Gamma$ is a cocompact lattice in $G$.
Then $G$ is amenable if and only if $\Gamma$ is amenable.
\end{corollary}

\begin{acknow}
I thank the Emil Aaltonen Foundation for support 
and Eeva-Maija Puputti for inspiration. 
\end{acknow}

\end{document}